\documentclass[a4paper,12pt,reqno,oneside]{amsart}
\usepackage[utf8]{inputenc}
\usepackage{amsmath,amssymb}
\usepackage[numbers]{natbib}
\usepackage{dsfont}
\usepackage{array}
\usepackage{booktabs}
\usepackage[font=small,labelfont=bf,format=hang]{caption}
\usepackage{enumitem}
\usepackage{indentfirst}
\usepackage[a4paper,top=3cm,bottom=3cm,left=3cm,right=3cm,headsep=18pt]{geometry}
\usepackage{graphicx}
\usepackage{randtext}
\usepackage{url}
\usepackage{hyperref}

\theoremstyle{plain}
\newtheorem{teo}{Theorem}[section]

\newtheorem{lem}{Lemma}[section]

\theoremstyle{definition}
\newtheorem{defn}{Definition}[section]

\theoremstyle{remark}
\newtheorem*{obs}{Remark}

\numberwithin{equation}{section}
\newcolumntype{P}[2]{>{#1\arraybackslash}m{#2}}

\newlist{lteo}{enumerate}{1}
\setlist[lteo,1]{font=\upshape,label=(\roman*)}

\newlist{lprova}{enumerate}{1}
\setlist[lprova,1]{leftmargin=*,font=\upshape,label=(\roman*)}

\newcommand{\bbZ}{\mathbb{Z}}

\newcommand{\bbT}{\mathbb{T}}

\newcommand{\GG}{\mathbb{G}}

\newcommand{\dsE}{\mathds{E}}
\newcommand{\dsP}{\mathds{P}}

\newcommand{\raiz}{\varnothing}
\newcommand{\dist}{\textnormal{dist}}

\newcommand{\YY}{\mathcal{X}}

\newcommand{\FM}[2]{\ensuremath{\textnormal{FM}(#1,#2)}}

\newcommand{\pc}[1]{\ensuremath{{p_c}(#1)}}

\newcommand{\s}[2]{\ensuremath{#1 \rightarrow #2}}
\newcommand{\ns}[2]{\ensuremath{#1 \nrightarrow #2}}

\newcommand{\p}[3]{\ensuremath{#1 \stackrel {#3}{\rightarrow} #2}}
\newcommand{\np}[3]{\ensuremath{#1 \stackrel {#3}{\nrightarrow} #2}}
\newcommand{\m}[3]{\ensuremath{#1 \stackrel {#3}{\leadsto} #2 }}

\begin{document}

\title[Frogs on homogeneous trees]{A new upper bound for the critical probability of the frog model on homogeneous trees}

\author{Elcio Lebensztayn}
\address[E. Lebensztayn and J. Utria]{Institute of Mathematics, Statistics and Scientific Computation\\
University of Campinas -- UNICAMP\\
Rua S\'ergio Buarque de Holanda 651, 13083-859, Campinas, SP, Brazil.}
\thanks{The authors were supported by the National Council for Scientific and Technological Development -- CNPq.}

\author{Jaime Utria}

\date{}

\subjclass[2010]{60K35, 60J85, 82B26, 82B43}

\keywords{Frog model, homogeneous tree, critical probability}

\begin{abstract}
We consider the interacting particle system on the homogeneous tree of degree $(d + 1)$, known as frog model.
In this model, active particles perform independent random walks, awakening all sleeping particles they encounter, and dying after a random number of jumps, with geometric distribution.
We prove an upper bound for the critical parameter of survival of the model, which improves the previously known results.
This upper bound was conjectured in a paper by \citeauthor{IUB} (\emph{J.\ Stat.\ Phys.}, 119(1-2), 331--345, 2005).
We also give a closed formula for the upper bound.
\end{abstract}

\maketitle

\baselineskip=22pt

\section{Introduction}
\label{S: Intro}

We study a system of branching random walks with random lifetime on a rooted graph, which is known as the frog model.
This model is inspired by the process of an infection spreading (or rumor propagation) through a population, where the transmission agents are mobile particles which move along the vertices of the graph.
Like the percolation model and other interacting particle systems such as the contact process, the frog model has two different types of behavior, depending on a parameter ($p$, in this case).
For small values of $p$, the process dies out almost surely, whereas, for large~$p$, the process survives perpetually with positive probability.
In this paper, we derive a new upper bound for the critical value $p_c$ that separates
these two regimes for the model on homogeneous trees.
This upper bound has been conjectured by \citet{IUB}.
For the proof, we construct an approximating sequence of processes that lead directly to the upper bound.

The \textit{frog model with geometric lifetime} is defined as follows.
Let $\GG$ be an infinite connected rooted graph.
The root of $\GG$ is denoted by $\raiz$.
We will assume that the system starts with one active particle located at $\raiz$, and one inactive particle at every other vertex of $\GG$.
Active particles perform independent simple random walks on~$\GG$, in discrete time, and have independent geometrically distributed random lifetimes, with parameter $(1 - p)$.
That is, at each instant of time, every awake particle may disappear with probability $(1-p)$; if it survives, it moves to one of the neighboring vertices, chosen with uniform probability.
Whenever an active particle hits a vertex containing a sleeping one, the latter is activated, and starts its own life and trajectory on $\GG$, in an independent manner.
We denote by $\FM{\GG}{p}$ the frog model on $\GG$, with survival parameter $p$, referring the reader to \citet{PT} for the formal definition of the model.

The question of phase transition for the frog model on infinite graphs was first addressed by \citet{PT}, especially on the hypercubic lattices $\bbZ^d$, $d \geq 1$, and homogeneous trees.
\citet{Mono} prove that the critical probability is not a monotonic function of the underlying graph, which is an unexpected fact, as the frog model can be thought of as a percolation model.
\citet{PTBT} study the issue of phase transition on birregular trees.
We present the known results regarding upper bounds for the critical parameter on homogeneous trees in Section~\ref{SS: DO}.

For the frog model with $p = 1$, in which all awake frogs live perpetually, a fundamental problem is whether the root of the graph is visited by infinitely many frogs.
In the first published paper dealing with the frog model, \citet{TW} prove that this holds true almost surely on $\bbZ^{d}$, for every $d \geq 1$, when the process starts with one particle per vertex.
\citet{FIRE} studies a phase transition from transience to recurrence, also on $\bbZ^{d}$, with respect to the initial density of particles.
The question of recurrence or transience for the frog model on integer lattices and infinite trees is an important topic of current research; see \citet{DGHPW,JJHb,KZ,JR}, and references therein.
Other central problems for the model without death are related to the growth of the set of visited vertices, and the movement of the cloud of particles.
We refer, for instance, to \citet{ST}, \citet{HW}, and \citet{HJJ-IS}.

\subsection{Definitions and objectives}
\label{SS: DO}

For $d \geq 2$, let $\bbT_{d}$ denote the homogeneous tree of degree $(d+1)$.
We say that a particular realization of the frog model on $\bbT_{d}$ \textit{survives} if for every instant of time there is at least one awake particle.
Otherwise, we say that it \textit{dies out}.
By a coupling argument, $\dsP[\FM{\bbT_{d}}{p}\;\text{survives}]$ is a nondecreasing function of $p$, and therefore we define the \textit{critical probability} as
\[\pc{\bbT_{d}} = \inf\left\{p: \dsP[\FM{\bbT_{d}}{p} \text{ survives}] > 0\right\}.\]

As usual, we say that there is \textit{phase transition} if $ \pc{\bbT_{d}} \in (0, 1)$.
As proved by \citet[Theorems~1.2 and 1.5]{PT}, the frog model on $\bbT_d$ with random initial configuration exhibits phase transition for every $d\geq 2$, under rather broad conditions.

The main result of \citet[Theorem~4.1]{IUB} states that, for the frog model on $\bbT_{d}$ starting with one particle per vertex,
\begin{equation}
\label{F: OrUpB}
p_c(\bbT_{d}) \leq \frac {d + 1}{2 d}.
\end{equation}
In a remark at the end of Section~4 (p.~341), the authors claim that a refinement of the argument in the proof of~\eqref{F: OrUpB} leads a better result, namely,
\begin{equation}
\label{F: UpB}
\pc{\bbT_{d}}\leq \frac{(d+1)\bar v}{1+ d\bar v^{2}},
\end{equation}
where $\bar v$ is the unique root in the interval $[0, 1/d]$ of the quartic polynomial $R^{(d)}(v)=d^{2}v^{4}-d(d+1)v^{3}+2d v-1$.
However, as the authors point out, some technical difficulties prevent a full proof of~\eqref{F: UpB}.
Our foremost purpose is to establish this new upper bound for $\pc{\bbT_{d}}$, which also improves the result derived recently by \citet{FMRT}, using Renewal Theory.
We present a proof of~\eqref{F: UpB} that relies on the ideas of~\citet{IUB}, but is completely independent, and has the advantage of offering a probabilistic interpretation, since it defines the approximation processes that allow us to arrive at formula \eqref{F: UpB}.

\section{Main results}
\label{S: MR}

For $d \geq 2$, we define
\begin{equation*}
\bar p(d) = \frac{(d+1)\bar v}{1+ d\bar v^{2}},
\end{equation*}
where $\bar v = \bar v(d)$ is the unique root in $[0, 1/d]$ of the polynomial
\begin{equation}
\label{F: Pol-v}
R^{(d)}(v)=d^{2}v^{4}-d(d+1)v^{3}+2d v-1.
\end{equation}
Let us denote by $\hat p(d)$ the upper bound for $\pc{\bbT_{d}}$ established in \citet[Proposition~2]{FMRT}:
\begin{equation*}
\hat p(d) =
\left\{
\begin{array}{cl}
0.720836 &\text{if } d = 2, \\[0.2cm]
\dfrac{(d+1)[(7d-1)-\sqrt{(7d-1)^{2}-14}]}{d(7d-1)^{2}-7d+2-d(7d-1)\sqrt{(7d-1)^{2}-14}}
&\text{if } d \geq 3.
\end{array}	\right.
\end{equation*}
It is worth noting that this bound improves~\eqref{F: OrUpB}.

\begin{teo}
\label{T: LS}
For every $d \geq 2$, we have that
\begin{lteo}
\item $\pc{\bbT_{d}} \leq \bar p(d)$.
\item $\bar p(d)$ is the unique root in (0,1) of the polynomial
\begin{equation}
\label{F: Pol-p}
Q^{(d)}(p)=p^4-\frac{4(d+1)}{(3d+1)} p^3-\frac{2(d-1)(d+1)^2}{(3d+1)^2} p^2+\frac{(d+1)^3}{d (3 d+1)} p-\frac{(d+1)^4}{d(3d+1)^2}.
\end{equation}
\item $\bar p(d)<\hat p(d)$.
\end{lteo}
\end{teo}

For the sake of completeness, we write down an explicit formula for $\bar p(d)$, not involving imaginary numbers.
For this, we need the following definition.

\begin{defn}
\label{D: Args Pol}
For every fixed $d \geq 2$, we define the constants
{\allowdisplaybreaks
\begin{align*}
\mathcal{Q}&=-\frac{2(d+1)^2(d+2)}{(3d+1)^2},\\[0.1cm]
\mathcal{R}&=\frac{(d+1)^3(5d^2+2d+1)}{d(3d+1)^3},\\[0.1cm]
\mathcal{O}&=-\frac{(d-1)^2 (d+1)^6 \left(16 d^3-259 d^2-162 d-27\right)}{3456 \, d^2 (3 d+1)^6},\\[0.1cm]
\mathcal{P}&=-\frac{ (d-1)^2 (d+1)^4}{36(3 d+1)^4},\\[0.1cm]
\Theta &=\frac{1}{3}\arccos\left(\frac{\mathcal{O}}{\sqrt{-\mathcal{P}^3}}\right),
\; \text{with} \; \arccos:[-1,1] \mapsto [0,\pi].
\end{align*}}%
Also, let
{\allowdisplaybreaks
\begin{equation*}
\mathcal{K}=\left\{
\begin{array}{cl}
6^{-1/2}[-\mathcal{Q} + 6\{(\mathcal{O}+\sqrt{\mathcal{O}^2 +\mathcal{P}^3})^{1/3}+(\mathcal{O}-\sqrt{\mathcal{O}^2 +\mathcal{P}^3})^{1/3}\}]^{1/2} & \text{if } 2 \leq d \leq 9,\\[0.4cm]
6^{-1/2}[-\mathcal{Q}+12\sqrt{-\mathcal{P}}\cos(\Theta)]^{1/2} & \text{if } d \geq 10.
\end{array} \right.
\end{equation*}}%
\end{defn}

\begin{teo}
\label{T: Root Pol UB}
For every $d\geq 2$,
\begin{equation*}
\bar p(d)=\frac{d+1}{3d+1}- \mathcal{K}+\frac{[\mathcal{K}(-4 \mathcal{K}^3-2\mathcal{K} \mathcal{Q}+\mathcal{R})]^{1/2}}{2\mathcal{K}}.
\end{equation*}
\end{teo}

\section{Proof of Theorem~\ref{T: LS}}
\label{S: Proof LS}

The key idea is to construct a class of Galton--Watson branching processes which are dominated by the frog model on $\bbT_d$, in the sense that the frog model survives if each one of these processes does.
This approach is similar to that employed in~\citet{IUB} to prove~\eqref{F: OrUpB}.
However, here the branching processes are defined in a new manner, in such a way that, by studying their critical behavior, we obtain directly a sequence of upper bounds for the critical probability, which converges to $\bar p(d)$.

Let $d \geq 2$ be fixed.
The first step in the proof is to describe $\FM{\bbT_{d}}{p}$ as a percolation model.
Indeed, for each pair $(x,y)$ of distinct vertices of $\bbT_{d}$, we draw a directed edge from~$x$ to~$y$ if and only if the particle placed originally at $x$ ever visits the vertex $y$, in the event of being activated.
As usual, we denote this event by $[\s{x}{y}]$, and its complement by $[\ns{x}{y}]$.
Of course, the frog model survives if and only if there exists an infinite sequence of distinct vertices $x_0 = \raiz, x_1, \dots$, satisfying $x_0 \rightarrow x_1 \rightarrow \cdots$ (that is, the cluster of the root in the oriented percolation model has infinite size).
As proved by \citet[Lemma 2.1]{IUB}, the probability of $[\s{x}{y}]$ is
\begin{equation}
\label{F: Prob open edge}
\dsP[\s{x}{y}] = (\beta^{(d)}(p))^{\dist(x,y)},
\end{equation}
where $\dist(x,y)$ is the distance between vertices $x$ and $y$, and the function $\beta^{(d)}:[0, 1]\rightarrow[0, 1/d]$ is given by
\begin{equation*}
\beta^{(d)}(p)=\left\{
\begin{array}{cl}
\dfrac{(d+1)-\sqrt{(d+1)^{2}-4d p^{2}}}{2d p} & \text{if } 0 < p \leq 1, \\[0.5cm]
0 & \text{if } p = 0.
\end{array} \right.
\end{equation*}
We also recall from Lemma~3.1 of this paper that, for every $v\in[0,1/d]$,
\begin{equation}
\label{F: Iff}
\beta^{(d)}(p)=v \iff p=\frac{(d+1)v}{1+d v^{2}}.
\end{equation}

In the sequel, we need the following definition.

\begin{defn}
\label{D: Func}
For $b \in [0, 1/d]$, we define
\begin{align*}
\psi(b) &= \sqrt{b^4-4b+4}, \quad \text{and} \\[0.2cm]
\lambda(b) &= \frac{b}{2} \left(2-b^2 + \psi(b)\right).
\end{align*}
\end{defn}

\noindent
The central idea to prove Theorem~\ref{T: LS} is summarized in the next result.

\begin{lem}
\label{L: BP}
For every $n \geq 1$, there exist a Galton--Watson branching process $\{X^{(d)}_{\ell,n}\}_{\ell \geq 0}$ embedded in $\FM{\bbT_{d}}{p}$, and a function $\phi_n$ (not depending on $d$) with domain $[0, 1/d]$, with the following properties:
\begin{lteo}
\item $X^{(d)}_{0,n} = 1$.
\item The survival of $\{X^{(d)}_{\ell,n}\}_{\ell \geq 0}$ implies the survival of the frog model on $\bbT_{d}$.
\item The mean number of offspring per individual in $\{X^{(d)}_{\ell,n}\}_{\ell \geq 0}$ equals
\begin{equation*}
\dsE[X^{(d)}_{1,n}]=d^n \, \phi_n(\beta^{(d)}(p)).
\end{equation*}
\end{lteo}
Furthermore, for every $b \in [0, 1/d]$,
\begin{equation}
\label{F: Lim}
\lim_{n \to \infty} {\left[ \phi_n(b) \right]}^{1 / n} = \lambda(b).
\end{equation}
\end{lem}

Notice that, by finding $p$ such that each one of these branching processes is supercritical, we obtain a sequence $\{\bar p_{n}(d)\}_{n \geq 1}$ of upper bounds for $\pc{\bbT_{d}}$.
Hence, to conclude that $\pc{\bbT_{d}} \leq \bar p(d)$, it will suffice to show that $\bar p_{n}(d)$ converges to $\bar p(d)$ as $n \to \infty$.
Before doing this, we present the construction of the branching processes and the proof of Lemma~\ref{L: BP}.

\subsection{Proof of Lemma~\ref{L: BP}}
\label{SS: BP}

To define the branching process $\{X^{(d)}_{\ell,n}\}_{\ell \geq 0}$, we consider the spreading of the frog model restricted to a tree rooted at $\raiz$, which is isomorphic to the $d$-ary tree.
First, given two vertices $x$ and $y$ of $\bbT_{d}$, we say that $y$ is a \textit{descendant} of $x$ if $x$ is one of the vertices of the path connecting $\raiz$ and $y$.
For $x\neq \raiz$, let $\bbT_{d}^{+}(x)$ denote the set consisting of all descendants of $x$
(including $x$ itself).
Fixed an arbitrary vertex $\raiz'$ neighbor of the root, we define $\bbT_{d}^{+}(\raiz)$ as the set of vertices in the subtree of $\bbT_{d}$ rooted at $\raiz$, that is obtained by disconnecting $\bbT_{d}^{+}(\raiz')$ from~$\bbT_{d}$.
See Figure~\ref{Fg: BPtreeH}.
For a vertex $x$ and $n \geq 1$, we denote by $L_n(x)$ the set of vertices in $\bbT_{d}^+(x)$ at distance $n$ from $x$.

\begin{figure}
\centering
\includegraphics[scale=0.7]{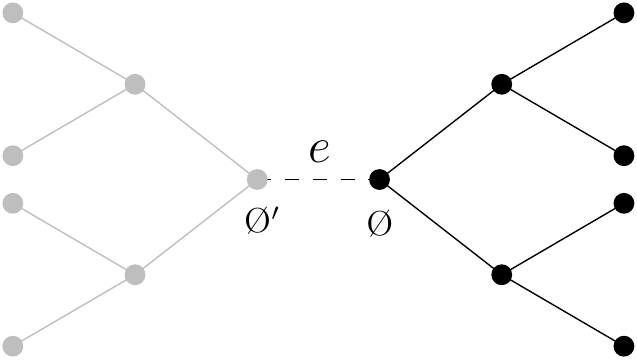}
\caption{$\bbT_{2}^{+}(\raiz)$ is depicted in black, and $\bbT_{2}(\raiz')$ in gray.
The edge~$e$ connecting $\raiz$ and $\raiz'$ is dashed.}
\label{Fg: BPtreeH}
\end{figure}

\begin{defn}
For vertices $x$ and $y\in L_{n}(x)$, let $x=x_0, x_1, \dots, x_{n}=y$ be the vertices in the path connecting $x$ and $y$, in such a way that $x_{\ell}$ and $x_{\ell+1}$ are neighbors.
For each $\ell=1, 2, \dots, n-1$, let $[\m{x_0}{x_{\ell}}{}]$ denote the event that $[\s{x_0}{x_{\ell}}] \cap [\ns{x_0}{x_{\ell+1}}]$.
Now we define the event $[\p{x_0}{x_{n}}{o}]$ inductively on $n$ by:
\begin{itemize}
\item[(i)] If $n=1$, then
\[[ \p{x_0}{x_1}{o}] = [\s{x_0}{x_1}]. \]
\item[(ii)] If $n=2$, then
\[[ \p{x_0}{x_2}{o}] = [\s{x_0}{x_2}] \cup [\m{x_0}{x_1}{},\p{x_1}{x_2}{o}]. \]
\item[(iii)] If $n\geq 3$, then
\[[ \p{x_0}{x_{n}}{o}] = [\s{x_0}{x_{n}}]\cup [\m{x_0}{x_1}{},\p{x_1}{x_n}{o}] \cup \bigcup_{\ell=2}^{n-1}[\m{x_0}{x_{\ell}}{},\p{\langle x_{\ell-1}, x_\ell \rangle}{x_{n}}{o}], \]
where
\[ [\p{\langle x_{\ell-1}, x_\ell \rangle}{x_{n}}{o}]=[\p{x_{\ell-1}}{x_n}{o}]\cup[\p{x_{\ell}}{x_n}{o}]. \]
\end{itemize}
Moreover, we denote the complement of $[\p{x}{y}{o}]$ by $[\np{x}{y}{o}]$.
\end{defn}

Fixed $n \geq 1$, let us define the Galton--Watson branching process $\{X^{(d)}_{\ell,n}\}_{\ell \geq 0}$ embedded in $\FM{\bbT_{d}}{p}$.
Consider $\YY^{(d)}_{0,n} = \{\raiz\}$, and, for $\ell \geq 1$, define
\[\YY^{(d)}_{\ell,n} = \bigcup_{x\in \YY_{\ell-1,n}}\{y\in L_{n}(x): \p{x}{y}{o}\}.\]
Let $X^{(d)}_{\ell,n}=|\YY^{(d)}_{\ell,n}|$ denote the cardinality of $\YY^{(d)}_{\ell,n}$.
Thus, the process starts from the root of $\bbT_{d}$, and, given a vertex $x$, its potential direct children are vertices located in $L_{n}(x)$.
A vertex $y\in L_{n}(x)$ is regarded as a \textit{child} of $x$ if and only if $[\p{x}{y}{o}]$.

Notice that, for every $n \geq 1$, $\{X^{(d)}_{\ell,n}\}_{\ell \geq 0}$ is indeed a branching process with $X^{(d)}_{0,n} = 1$, and whose survival implies the survival of the frog model on $\bbT_{d}$.
To finish the proof of Lemma~\ref{L: BP}, we need to show part~(iii) and formula~\eqref{F: Lim}.
To this end, we define the sequence of functions $\{\phi_n\}_{n\geq 1}$ with domain $[0, 1/d]$, inductively given by
\begin{equation}
\label{F: DI}
{\allowdisplaybreaks
\begin{aligned}
\phi_1(b)&=b,\\[0.1cm]
\phi_2(b)&=b^2(2-b),\\[0.1cm]
\phi_n(b)&=b^n+b(1-b)\phi_{n-1}(b)+{}\\[0.1cm]
&\phantom{=}{}+\sum_{\ell=2}^{n-1} b^\ell(1-b)[\phi_{n-\ell+1}(b)+(1-b)\phi_{n-\ell}(b)], \, n\geq 3.
\end{aligned}}%
\end{equation}
Then, part (iii) of Lemma~\ref{L: BP} is established once we prove the following result.

\begin{lem}
\label{L: Prob desc}
For every $n \geq 1$, the probability of vertex $x_{n} \in L_{n}(x_{0})$ being a child of~$x_{0}$ is given by
\[ \dsP[\p{x_0}{x_n}{o}] = \phi_n(\beta^{(d)}(p)), \]
where $\phi_{n}$ is defined by \eqref{F: DI}.
\end{lem}

\begin{proof}
First, from \eqref{F: Prob open edge}, it follows that for all $\ell=1, \dots, n-1$,
\[ \dsP[\m{x_0}{x_\ell}{}] = (\beta^{(d)}(p))^\ell \, [1-\beta^{(d)}(p)]. \]
In addition, notice that $\dsP[\p{\langle x_{\ell-1}, x_\ell \rangle}{x_{n}}{o}]$ can be obtained using the Inclusion-Exclusion Formula, and
\[\dsP[\p{x_{\ell-1}}{x_n}{o},\p{x_{\ell}}{x_n}{o}] =
\dsP[\s{x_{\ell-1}}{x_\ell}] \, \dsP[\p{x_\ell}{x_{n}}{o}] =
\beta^{(d)}(p) \, \dsP[\p{x_\ell}{x_{n}}{o}].\]
Hence, observing that for $n \geq 3$,
\begin{align*}
\dsP[\p{x_0}{x_n}{o}]&=\dsP[\s{x_0}{x_n}]+
\dsP[\m{x_0}{x_1}{}] \, \dsP[\p{x_{1}}{x_{n}}{o}]+{}\\[0.1cm]
&\phantom{=}{}+\sum_{\ell=2}^{n-1}\dsP[\m{x_0}{x_\ell}{}] \, \dsP[\p{\langle x_{\ell-1}, x_\ell \rangle}{x_{n}}{o}],
\end{align*}
and using \eqref{F: Prob open edge}, the result follows by induction on $n$.
\end{proof}

The proof of formula \eqref{F: Lim} proceeds in two steps, which are formulated in Lemmas \ref{L: Form Recorr} and \ref{L: Lim Recorr}.

\begin{lem}
\label{L: Form Recorr}
The sequence $\{\phi_n\}_{n\geq 1}$ satisfies the following linear difference equation of second order
\begin{equation}
\label{F: Recorr}
\phi_{n}(b)=b(2-b^2)\phi_{n-1}(b)-b^3(1-b)\phi_{n-2}(b), \, n \geq 3,
\end{equation}
with initial conditions $\phi_1(b)=b$ and $\phi_2(b)=b^2(2-b).$
\end{lem}

\begin{proof}
It is straightforward to check that, for $n = 3$, both formulas \eqref{F: DI} and \eqref{F: Recorr} lead to $\phi_3(b) = b^3 (4 - 3 b - b^2 + b^3)$.
From~\eqref{F: DI}, it follows that, for every $n \geq 3$,
{\allowdisplaybreaks
\begin{align*}
\phi_{n+1}(b)&=b^{n+1}+\sum_{\ell=2}^{n} b^\ell(1-b)[\phi_{n-\ell+2}(b)+(1-b)\phi_{n-\ell+1}(b)]+b(1-b)\phi_{n}(b)\\
&=b^{n+1}+\sum_{\ell=2}^{n-1} b^{\ell+1}(1-b)[\phi_{n-\ell+1}(b)+(1-b) \phi_{n-\ell}(b)]+{}\\
&\phantom{=}{}+b^2(1-b)[\phi_n(b)+(1-b)\phi_{n-1}(b)]+b(1-b)\phi_n(b).\\
\intertext{Thus, pulling a factor of $b$ out of the first two addends, we obtain}
\phi_{n+1}(b)&=b[\phi_n(b)-b(1-b)\phi_{n-1}(b)]+{}\\
&\phantom{=}{}+b^2(1-b)[\phi_n(b)+(1-b)\phi_{n-1}(b)]+b(1-b)\phi_n(b)\\
&=b(2-b^2)\phi_n(b)-b^3(1-b)\phi_{n-1}(b),
\end{align*}}%
as desired.
\end{proof}

\begin{lem}
\label{L: Lim Recorr}
For every $b \in [0, 1/d]$, we have that
\[ \lim_{n \to \infty} {\left[ \phi_n(b) \right]}^{1 / n} = \lambda(b)
= \frac{b}{2} \left(2-b^2 + \psi(b)\right). \]
\end{lem}

\begin{proof}
It is enough to consider $b \in (0, 1/d]$.
In order to solve \eqref{F: Recorr}, we compute the discriminant $\Delta_0$ of the associated characteristic polynomial
\begin{equation}
\label{F: EC}
\lambda^2-b(2-b^2)\lambda+b^3(1-b)=0,
\end{equation}
obtaining $\Delta_0=b^2(4-4b+b^4) > 0$.
Therefore, the equation \eqref{F: EC} has two distinct real roots, which are given by
\[ \lambda_{\pm}(b) = \frac{b}{2} \left(2-b^2 \pm \psi(b)\right). \]
The general solution of \eqref{F: Recorr} is then expressed as
\begin{equation}
\label{F: SG}
\phi_n(b) = c_1 \, (\lambda_{-}(b))^n + c_2 \, (\lambda_{+}(b))^n
= (\lambda_{+}(b))^n \left[ c_1 \left( \frac{\lambda_{-}(b)}{\lambda_{+}(b)} \right)^n + c_2 \right],
\end{equation}
where $c_1$ and $c_2$ are functions of $b$, determined by the initial conditions.
As $\lambda(b) = \lambda_{+}(b) > \lambda_{-}(b)$, the result follows.
\end{proof}

\begin{obs}
To derive a closed formula for $\phi_n$, we can extend the sequence $\{\phi_n\}$ for the index $n=0$, by rewriting
\begin{equation*}
\phi_{n}(b)=b(2-b^2)\phi_{n-1}(b)-b^3(1-b)\phi_{n-2}(b), \, n \geq 2,
\end{equation*}
with initial conditions $\phi_0(b)=1$ and $\phi_1(b)=b$.
By imposing these initial conditions in~\eqref{F: SG}, we get
\[ c_1 = \frac{\psi(b)-b^2}{2 \, \psi(b)} \quad \text{and} \quad
c_2 = \frac{\psi(b)+b^2}{2 \, \psi(b)}. \]
Consequently, for every $n \geq 0$ and $b \in [0, 1/d]$,
\begin{equation*}
\phi_n(b)=\frac{b^n}{2^{n+1} \, \psi(b)} \left[(\psi(b)-b^2)(2-b^2-\psi(b))^n+(\psi(b)+b^2)(2-b^2+\psi(b))^n\right].
\end{equation*}
We underline that $\phi_n$ equals the function $G_n$ obtained through the alternative method described by \citet{IUB}, in the remark at the end of Section~4, for constructing another approximating function $G_n$ to $F_n$, such that $F_n(b) \geq G_n(b)$ for every $b$.
That is, $\phi_n$ is identical to the function $G_n$ that one obtains when the inequality $(1-b^k) \geq (1-b^2)$ for $k \geq 2$ is used instead of $(1-b^k) \geq (1-b)$ for $k \geq 1$, in the definition of $G_n$.
This is the reason why the upper bound we establish here coincides with the one conjectured by \citet{IUB}.
\end{obs}

\subsection{Proof of Theorem \ref{T: LS}}
\label{SS: Proof LS}

Again let $d \geq 2$ be fixed.
We define the functions
\begin{equation}
\label{F: fn/f}
f_n^{(d)}(p) = {\left[ \phi_n(\beta^{(d)}(p)) \right]}^{1 / n} - \frac{1}{d}, \, n \geq 1,
\quad \text{and} \quad
f^{(d)}(p) = \lambda(\beta^{(d)}(p)) - \frac{1}{d}.
\end{equation}
From Equation~\eqref{F: Lim}, we have that, for every $p \in [0, 1]$,
\begin{equation*}
\lim_{n\to \infty} f_n^{(d)}(p) = f^{(d)}(p).
\end{equation*}

We also observe that
\begin{equation}
\label{F: Equiv}
\lambda(b) = \frac{1}{d} \iff R^{(d)}(b) = 0,
\end{equation}
where $R^{(d)}$ is the quartic polynomial given in~\eqref{F: Pol-v}.
Since $\lambda$ is an increasing function, satisfying $\lambda(0) = 0$ and $\lambda(1/d) > 1/d$ (as $R^{(d)}(1/d) = (d - 1)/d > 0$), we conclude that the equivalent equations in~\eqref{F: Equiv} have a unique solution $\bar v = \bar v(d)$ in the interval $[0, 1/d]$.
In addition, from \eqref{F: Iff}, it follows that
\[ \bar p(d) = \frac{(d+1)\bar v}{1+ d\bar v^{2}} \]
is the unique root of the equation $f^{(d)}(p) = 0$ in the interval $[0, 1]$.

To prove part (i) of Theorem \ref{T: LS}, we use the following fact of Real Analysis.

\begin{lem}
\label{L: Conv}
Let $\{ f_n \}$ be a sequence of increasing, continuous real-valued functions defined on $[0, 1]$, such that $f_n (0) < 0$ and $f_n (1) > 0$ for every $n$.
Suppose that $\{ f_n \}$ converges pointwise as $n \to \infty$ to an increasing, continuous function $f$ defined on $[0, 1]$, and let $\bar r_n$ be the unique root of $f_n$ in $[0, 1]$.
Then, there exists $\bar r = \lim_{n \to \infty} \bar r_n$ and $f (\bar r) = 0$.
\end{lem}

It is straightforward to check that $\{f_n^{(d)}\}$ and $f^{(d)}$ given in \eqref{F: fn/f} satisfy the conditions of Lemma \ref{L: Conv}.
Hence, by defining $\bar p_{n} = \bar p_{n}(d)$ as the unique root of $f_n^{(d)}$ in the interval 
$[0, 1]$, we obtain that
\[ \lim_{n\to\infty} \bar p_{n}(d) = \bar p(d). \]
But for every $p > \bar p_{n}(d)$, we have that $f_n^{(d)}(p) > 0$, whence, from Lemma \ref{L: BP}, the frog model on $\bbT_{d}$ survives with positive probability.
Consequently,
\[ p_c(\bbT_{d}) \leq \bar p_n(d), \]
and part (i) of Theorem~\ref{T: LS} is established by taking $n \to \infty$.

Part (ii) is proved by expanding and simplifying properly the equation 
\[ R^{(d)}(\beta^{(d)}(p)) = 0. \]
The formula manipulation involved can be accomplished by using a mathematical software.
This tool can also be used to establish that $Q^{(d)}(\hat p (d)) > 0$ for every $d\geq 2$.
Together with the fact that $Q^{(d)}(0) < 0$, this yields part (iii).~\qed

\section{Proof of Theorem \ref{T: Root Pol UB}}
\label{S: Proof Root}

In brief, the proof proceeds by using Descartes' solution of a quartic equation and a couple of results to isolate the root in $(0, 1)$ among the four roots of $Q^{(d)}$. 
The central steps are explained below.
See \citet{Dickson} for an excellent account on the theory of equations.

First, notice that, by Descartes' Rule of Signs, the polynomial $Q^{(d)}$ given in \eqref{F: Pol-p} has either one or three positive roots. 
Naturally, $\bar p(d)$ is the smallest positive root. 
Using Descartes' method, we apply the Tschirnhaus transformation $p = z + (d+1)/(3 d+1)$ to the equation $Q^{(d)}(p) = 0$, thereby obtaining the reduced form
\begin{equation}
\label{F: pol red}
z^4+\mathcal{Q}z^2+\mathcal{R}z+\mathcal{S}=0,
\end{equation}
where $\mathcal{Q}$ and $\mathcal{R}$ are given in Definition \ref{D: Args Pol}, and $\mathcal{S}=-(d+1)^4(2d+1)/(3d+1)^4$. 
Next we consider the auxiliary cubic equation
\begin{equation}
\label{F: pol cub}
x^3+\frac{1}{2} \mathcal{Q} x^2+\frac{1}{16}(\mathcal{Q}^2-4\mathcal{S})x-\frac{1}{64}\mathcal{R}^2 = 0.
\end{equation}
To solve equation \eqref{F: pol red}, we have to pick any nonzero root of \eqref{F: pol cub}, say $\mathcal{W}$, and define $\mathcal{K}$ as either square root of the selected $\mathcal{W}$.
Then, the four roots of \eqref{F: pol red} are the roots of the two quadratic polynomials
{\allowdisplaybreaks
\begin{align*}
g_{1}(z)&=z^2+2\mathcal{K}z+\frac{1}{2}\mathcal{Q}+2\mathcal{K}^2-\frac{\mathcal{R}}{4\mathcal{K}},\\[0.2cm]
g_{2}(z)&=z^2-2\mathcal{K}z+\frac{1}{2}\mathcal{Q}+2\mathcal{K}^2+\frac{\mathcal{R}}{4\mathcal{K}},
\end{align*}}%
which are given by
\begin{equation}
{\allowdisplaybreaks
\begin{aligned}
\label{F: Roots}
z_1&= -\mathcal{K}-\frac{[\mathcal{K}(-4 \mathcal{K}^3-2\mathcal{K} \mathcal{Q}+\mathcal{R})]^{1/2}}{2\mathcal{K}},\\[0.2cm]
z_2&= -\mathcal{K}+\frac{[\mathcal{K}(-4 \mathcal{K}^3-2\mathcal{K} \mathcal{Q}+\mathcal{R})]^{1/2}}{2\mathcal{K}},\\[0.2cm]
z_3&= \mathcal{K}-\frac{[-\mathcal{K}(4 \mathcal{K}^3+2\mathcal{K} \mathcal{Q}+\mathcal{R})]^{1/2}}{2\mathcal{K}},\\[0.2cm]
z_4&= \mathcal{K}+\frac{[-\mathcal{K}(4 \mathcal{K}^3+2\mathcal{K} \mathcal{Q}+\mathcal{R})]^{1/2}}{2\mathcal{K}}.
\end{aligned}}%
\end{equation}

To study the roots of \eqref{F: pol red}, we compute the discriminant $\Delta$ of the equation \eqref{F: pol cub}.
This can be written as $\Delta=d^{-4} \, (3 d+1)^{-10} \, \mathcal{H}_0(d)/4096$, where
\[ \mathcal{H}_0(d)=(d-1)^4 (d+1)^{12} \left(32 d^3-275 d^2-162 d-27\right). \]
Notice that $\mathcal{H}_0(d)$ is a polynomial of degree $19$. 
Using Budan--Fourier Theorem for $\mathcal{H}_0(d)$ and the fact that $\mathcal{H}_0(9)<0<\mathcal{H}_0(10)$, we conclude that $\Delta<0$ for $2\leq d \leq 9$, and $\Delta>0$ for $d\geq 10$.
Hence, we split the proof into two cases:

\begin{lprova}
\item $2 \leq d \leq 9$:
We have that $\Delta < 0$, so the equation \eqref{F: pol red} has two distinct real and two imaginary roots, and therefore the roots of $Q^{(d)}$ are one negative, one positive and two complex conjugate to each other.
Moreover, the equation \eqref{F: pol cub} has a unique real root, which is positive. 
The value of $\mathcal{K}$ given in Definition~\ref{D: Args Pol} for $2 \leq d \leq 9$ equals the positive square root of this unique real root of equation \eqref{F: pol cub}, which is obtained using Cardano's solution for cubic equations.
We observe that $z_3$ and $z_4$ in \eqref{F: Roots} are imaginary numbers (otherwise $z_3$ would be negative). 
Consequently, $z_2+(d+1)/(3 d+1)$ is the unique positive root of $Q^{(d)}$.
This completes the proof of the formula for $\bar p(d)$ in case~(i).

\item $d \geq 10$:
As $\Delta>0$, $\mathcal{Q}<0$ and $4 \mathcal{S}-\mathcal{Q}^2<0$, the roots of \eqref{F: pol red} are all real and distinct, and this implies that $Q^{(d)}$ has one negative and three positive roots. 
Furthermore, the equation \eqref{F: pol cub} has three positive real roots, which are better expressed in a trigonometric form (the so-called irreducible case). 
The constant $\mathcal{K}$ as given in Definition~\ref{D: Args Pol} for $d \geq 10$ is the positive square root of one of these roots. 
Thus, the four roots of \eqref{F: pol red} are $z_1<z_2$ and $z_3<z_4$. 
To finish the proof, it is enough to show that $z_1<-(d+1)/(3d+1)$. 
But this inequality holds true, since $g_{1}(-(d+1)/(3d+1))<0$.~\qed
\end{lprova}


\bibliography{bibUBFH}
\bibliographystyle{plainnat}

\end{document}